\documentclass[letterpaper, 10 pt, conference]{ieeeconf}  

\IEEEoverridecommandlockouts                              

\usepackage{xcolor}
\usepackage{ntheorem}
\usepackage{hyperref}
\usepackage{textcomp}
\let\labelindent\relax
\usepackage{enumitem}
\usepackage{graphicx}
\usepackage{epstopdf}
\usepackage{float}
\usepackage[nocomma]{optidef}
\usepackage{dsfont}
\usepackage{times}
\usepackage{bm}
\usepackage{cancel}

\usepackage{fancyhdr,graphicx,amsmath,amssymb}
\usepackage{algorithm} 
\usepackage{algpseudocode} 

\usepackage[utf8]{inputenc}
\usepackage[T1]{fontenc}
\usepackage{color}
\usepackage{amsfonts}
\usepackage{amsmath}
\usepackage{amssymb}
\usepackage{mathtools}
\mathtoolsset{showonlyrefs}
\usepackage{commath}
\usepackage{float}
\usepackage{epsfig}
\usepackage{bm}
\usepackage{cite}
\usepackage{graphicx}
\usepackage{subfig}
\usepackage{tabularx}
\usepackage{longtable}
\usepackage{booktabs}

\newcommand{\grad}{{\nabla}}   
\newcommand{\one}{\mathbf{1}}   
\newcommand{\real}{\mathbb{R}}  

\newcommand{\diag}{{\rm diag}}  
\newcommand{\col}{\mathrm{col}}     

\newcommand{\tran}{^{\textit{\footnotesize \texttt{T}}}} 
 
\newcommand{\define}{\triangleq} 

\newcommand{\cut}[1]{%
{\leavevmode\color{black}#1}%
}

\def\bLambda{{\boldsymbol \Lambda}}
\def\bPhi{{\boldsymbol \Phi}}

\def\C{{\mathbf{C}}}

\def\G{{\mathbf{G}}}

\def\I{{\mathbf{I}}}

\def\Q{{\mathbf{Q}}}
\def\R{{\mathbf{R}}}

\def\W{{\mathbf{W}}}

\def\f{{\mathbf{f}}}
\def\g{{\mathbf{g}}}
\def\h{{\mathbf{h}}}

\def\x{{\mathbf{x}}}
\def\y{{\mathbf{y}}}

\newcommand{\cE}{{\mathcal{E}}}

\newcommand{\cG}{{\mathcal{G}}}

\newcommand{\cN}{{\mathcal{N}}}

\newcommand{\cV}{{\mathcal{V}}}

\newtheorem{theorem}{Theorem}
\newtheorem{lemma}[theorem]{Lemma}

\newtheorem{rem}{Remark}
\newtheorem{assumption}{Assumption}
\makeatletter
\let\NAT@parse\undefined
\makeatother
\usepackage{hyperref}
\hypersetup{
	colorlinks=true,
	linkcolor=red,     
	urlcolor=magenta,
	citecolor={blue},
}

\usepackage{algorithmicx,algpseudocode,algorithm}

\allowdisplaybreaks

\title{\LARGE \bf
On the Performance of Gradient Tracking with Local Updates
}

\author{Edward Duc Hien Nguyen, Sulaiman A. Alghunaim, Kun Yuan and C\'esar A. Uribe
\thanks{EDHN and CAU (\textit{\{en18,cauribe\}@rice.edu}) are with the Department of Electrical and Computer Engineering, Rice University, Houston, TX, USA. SAA (\textit{sulaiman.alghunaim@ku.edu.kw}) is with the Department of Electrical Engineering, Kuwait University, Kuwait City, Kuwait. KY (\textit{kun.yuan@alibaba-inc.com}) is with Alibaba DAMO Academy, Hangzhou, Zhejiang, China.}
\thanks{Edward Nguyen is supported by a training fellowship from the Gulf Coast Consortia, on the NLM Training Program in Biomedical Informatics \& Data Science~(T15LM007093). This work is supported in part by the National Science Foundation under Grant No. 2211815.}%
}

\begin{document}

\maketitle
\thispagestyle{empty}
\pagestyle{empty}

\begin{abstract}
We study the decentralized optimization problem where a network of $n$ agents seeks to minimize the average of a set of heterogeneous non-convex cost functions distributedly. State-of-the-art decentralized algorithms like Exact Diffusion~(ED) and Gradient Tracking~(GT) involve communicating every iteration. However, communication is expensive, resource intensive, and slow. In this work, we analyze a locally updated GT method (LU-GT), where agents perform local recursions before interacting with their neighbors. While local updates have been shown to reduce communication overhead in practice, their theoretical influence has not been fully characterized. We show LU-GT has the same communication complexity as the Federated Learning setting but allows arbitrary network topologies. In addition, we prove that the number of local updates does not degrade the quality of the solution achieved by LU-GT. Numerical examples reveal that local updates can lower communication costs in certain regimes (e.g., well-connected graphs).

\end{abstract}

\section{Introduction} \label{sec:Introduction}

Distributed optimization problems have garnered significant interest due to the demand for efficiently solving data processing problems~\cite{boyd2011admm}, such as the training of deep neural networks \cite{ying2021bluefog}. Nodes, processors, and computer clusters can be abstracted as agents responsible for a partition of a large set of data. { In this work}, we study the distributed multi-agent optimization problem 
\begin{equation} \label{eq:obj_function}
    \mathop{\text{minimize}}\limits_{x \in \real^m} \quad f(x) \define \frac{1}{n}\sum_{i=1}^n f_i(x)
\end{equation}
where $f_i(\cdot): \mathbb{R}^m \rightarrow \mathbb{R}$ is a smooth, non-convex function held privately by an agent $i \in \{1,\ldots,n\}$. To find a  consensual solution $x^*$ of~\eqref{eq:obj_function}, decentralized methods have the $n$ agents cooperate according to some network topology constraints.
\par
Many decentralized methods have been proposed to solve~\eqref{eq:obj_function}. Among the most prolific include decentralized/distributed gradient descent (DGD)~\cite{ram2010distributed, cattivelli2010diffusion}, EXTRA~\cite{shi2015extra},~Exact-Diffusion/D\textsuperscript{2}/NIDS (ED)~\cite{yuan2019exactdiffI, li2017nids, yuan2020influence, tang2018d}, and Gradient Tracking (GT)~\cite{xu2015augmented, di2016next, qu2017harnessing, nedic2017achieving}. DGD is an algorithm wherein agents perform a local gradient step followed by a communication round. \cut{However, DGD has been shown not optimal} when agents' local objective functions are heterogeneous, i.e., the minimizer of \cut{functions $f_i(\cdot)$ differs from the minimizer of $f(\cdot)$}. { This shortcoming has been analyzed in~\cite{chen2013distributed, yuan2016convergence} where the heterogeneity causes the rate of DGD to incur an additional bias term with a magnitude directly proportional to the level of heterogeneity, slowing down the rate of convergence. { Moreover, this bias term is inversely influenced by the connectivity of the network (becomes larger for sparse networks) \cite{yuan2020influence,koloskova2020unified}. } }
\par
{ EXTRA, ED, and GT employ bias-correction techniques to account for heterogeneity. EXTRA and ED use local updates with memory. GT methods have each agent perform the local update with an estimate of the global gradient called the tracking variable. The effect of these techniques is observed in the analysis where the rates of these algorithms are independent of heterogeneity, i.e., the bias term proportional to the amount of heterogeneity found in the rate of DGD is removed, { leading to better rates~\cite{alghunaim2021unified,koloskova2021improved}.}
}
\par
\cut{While EXTRA, ED, and GT tackle the bias induced by heterogeneity, they require communication over the network at every iteration.} However, communication is expensive, resource intensive, and slow in practice~\cite{ying2021bluefog}. Centralized methods { in which agents communicate with a central coordinator (i.e., server) have been developed to solve~\eqref{eq:obj_function} with an explicit focus on reducing the communication cost. This has been achieved empirically by requiring agents to perform local recursions before communicating.} Among these methods include LocalGD~\cite{stich2018local, khaled2019first, khaled2020tighter, zhang2016parallel, Lin2020Don't}, Scaffold~\cite{karimireddy2020SCAFFOLD}, S-Local-GD~\cite{gorbunov2021local}, FedLin~\cite{mitra2021linear}, and Scaffnew~\cite{mishchenko2022proxskip}. Analysis on LocalGD revealed that local recursions cause agents to drift towards their local solution rather than the global optimal~\cite{khaled2019first, khaled2019tighter, koloskova2020unified}. 
Scaffold, S-Local-GD, FedLin, and Scaffnew address this issue by introducing bias-correction techniques. However, besides Scaffnew, analysis of these methods has failed to show communication complexity improvements. { Scaffnew has shown that for $\mu$-strongly-convex, $L$-smooth, and deterministic functions, the communication complexity can be improved from $O(\kappa)$ (no local recursions) to $O(\sqrt{\kappa})$ if one performs $\sqrt{\kappa}$ local recursions with}  \mbox{\cut{$\kappa \define {L}/{\mu}$.}}
\par
Local recursions in {\em decentralized methods} have been much less studied. DGD with local recursions has been studied in~\cite{koloskova2020unified}, but the 
convergence rates still have bias terms due to heterogeneity. Additionally, the magnitude of the bias term is proportional to the number of local recursions taken. Scaffnew~\cite{mishchenko2022proxskip} has been studied under the decentralized case but for the strongly convex and smooth function class. { In ~\cite{mishchenko2022proxskip} for sufficiently well connected graphs, an improvement to a communication complexity of $O(\sqrt{\kappa/(1-\lambda}))$  where $\lambda$ is the mixing rate of the matrix is shown.} 
{
GT with local updates can be seen as a special case of the time-varying graph setting  (with no connections for several iterations, and the graph is connected periodically for one iteration). Several works studied GT under time-varying graphs such as~\cite{di2016next,nedic2017achieving,scutari2019distributed,sun2019convergence,saadatniaki2020decentralized}, among these only the works \cite{di2016next,scutari2019distributed,lu2020decentralized} considered nonconvex setting. Different from~\cite{di2016next,scutari2019distributed,lu2020decentralized}, we provide explicit expressions that characterize the convergence rate in terms of the problem parameters (e.g., network topology).  We note that no proofs are given in~\cite{lu2020decentralized}. Our analysis is also of independent interest and can be readily extended to stochastic costs or arbitrary time-varying networks. }
\par
In this work, we propose and study LU-GT, a locally updated decentralized algorithm based on the bias-corrected method GT. Our contributions are as follows:
\begin{itemize}[leftmargin=*]
    \item We analyze LU-GT under the deterministic, non-convex regime. { Our analysis provides a more convenient and simpler way to analyze GT methods, which builds upon and extends the framework outlined in~\cite{alghunaim2021unified}.} 
    \item { Our analysis has a communication complexity matching the rates for previously locally updated variants of centralized and distributed algorithms.}
    \item We demonstrate that LU-GT retains the bias-correction properties of GT irrespective of the number of local recursions and that the number of local recursions does not affect the quality of the solution.
    \item { Numerical analysis shows that local recursions can reduce the communication overhead in certain regimes, e.g., well-connected graphs.}
\end{itemize}
{
This paper is organized as follows. Section~\ref{sec:alg} defines relevant notation, states the assumptions used in our analysis, introduces LU-GT, and states our main result on the convergence rate. In Section~\ref{sec:prelims}, we provide intuition into how the direction of our analysis can show that following LU-GT, agents reach a consensus that is also a first-order stationary point. We also cover relevant lemmas needed in the analysis of LU-GT. In Section~\ref{sec:analysis}, we prove the convergence rate of LU-GT. Section~\ref{sec:numerical} shows evidence that the local recursions of LU-GT can reduce communication costs in certain regimes.}

\section{ Assumptions, Algorithm and Main Result} \label{sec:alg}

We begin this section by providing some useful notation:
\begin{subequations}
	\begin{gather}
        \x^k = x \otimes \mathbf{1}_n,\quad \y^k = y \otimes \mathbf{1}_n \\
		\f(\x){=}\sum_{i=1}^n f_i(x_i), \ \
		\grad \f(\x){=}\col\{\grad f_1(x_1),\dots,\grad f_n(x_n)\} .
	\end{gather}
\end{subequations}
We define $W = [w_{ij}] \in \real^{n \times n}$ as the symmetric mixing matrix for an undirected graph $\cG = \{\cV, \cE\}$ that models the connections of a group of $n$ agents. The weight $w_{ij}$ scales the information agent $i$ receives from agent $j$. We set $w_{ij} = 0$ if $j \notin \cN_i$ where $\cN_i$ is the set of neighbors of agent $i$. We also define $\mathbf{W} = W \otimes I_d \in \real^{mn \times mn}$.

\begin{algorithm}[t!]
\caption{LU-GT} \label{alg:LU_GT}
\begin{algorithmic}[1] 
\State Input: $\x^0 = \mathbf{0} \in \mathbb{R}^{mn}$, $\y^0 = \alpha \nabla \f(\x^0)$, $\alpha>0$, $\eta>0$ $T_{\rm o} \in \mathbb{Z}_{\geq 0}$, $K \in \mathbb{Z}_+$
\State Define: $\tau = \{0, T_{\rm o}, 2T_{\rm o}, 3T_{\rm o} ...\}$
\For{$k = 0,...,K-1$}
    \If {$k \in \tau$}
    \State $\x^{k+1} = \W (\x^k - \eta \y^k)$ \label{lst:line:comm_rec_1}
    \State $\y^{k+1} = \W (\y^k + \alpha \nabla \f(\x^{k+1}) - \alpha \nabla \f(\x^{k}))$
    \Else
    \State $\x^{k+1} = \x^k - \eta \y^k$ \label{lst:line:local_rec_1}
    \State $\y^{k+1} = \y^k + \alpha \nabla \f(\x^{k+1}) - \alpha \nabla \f(\x^{k})$
    \EndIf
\EndFor
\end{algorithmic}
\end{algorithm}

The proposed method LU-GT is detailed in Algorithm~\ref{alg:LU_GT} where $\alpha$ and { $\eta$ are step-size parameters}, and $T_{\rm o} - 1$ is the number of local recursions before a round of communication. The intuition behind the algorithm is to have agents perform a descent step using a staling estimate of the global gradient for $T_{\rm o} - 1$ iterations. { Afterwards, agents perform a weighted average of their parameters with their neighbors and update their tracking variable.}
{ 
\begin{rem}
For $T_{\rm o}=1$, Algorithm \ref{alg:LU_GT} becomes equivalent to the vanilla ATC-GT \cite{xu2015augmented} with stepsize $\bar{\eta}=\eta \alpha$. This can be seen by introducing the change of variable $\g^k=(1/\alpha) \y^k$. Thus, our analysis also covers the vanilla GT method.
\end{rem}
}
\par
{
To analyze Algorithm~\ref{alg:LU_GT}, we first introduce the following time-varying matrix:

\vspace{-0.5cm}
\begin{align} \label{def:W_k}
    \W_k \triangleq \begin{cases} \W & \text{when } k \in \tau, \\ \I & \text{otherwise.} \end{cases}
\end{align}

Thus, we can succinctly rewrite Algorithm~\ref{alg:LU_GT} as follows
\begin{IEEEeqnarray}{rCl}
\label{eq:GT_rec_1}
    \x^{k+1} &= &\W_k (\x^k - \eta \y^k) \IEEEyessubnumber\\
    \y^{k+1} &= &\W_k (\y^k + \alpha \nabla \f (\x^{k+1}) - \alpha \nabla \f (\x^{k})). \IEEEyessubnumber
\end{IEEEeqnarray}
}
We now list the assumptions used in our analysis.

\begin{assumption}[Mixing matrix]\label{ass:matrix}
The mixing matrix $W$ is doubly stochastic and symmetric.
\end{assumption}
\par
{ The Metropolis-Hastings algorithm~\cite{hastings1970metro} can be used to construct mixing matrices from an undirected graph satisfying Assumption~\ref{ass:matrix}}. Moreover, from Assumption~\ref{ass:matrix}, we have that the  mixing matrix $W$ has a singular, maximum eigenvalue denoted as $\lambda_1 = 1$. All other eigenvalues are defined as $\{\lambda_i\}^{n}_{i=2}$. We define the mixing rate as \mbox{$\lambda := \max_{i \in\{2,...,n\}} \{|\lambda_i|\}$.}

\begin{assumption}[$L$-smoothness]\label{ass:smooth} 
The function $f_i : \real^m \rightarrow \real$ is $L$-smooth for $i \in\mathcal{V}$, i.e., $\lVert \nabla f_i(y) - \nabla f_i (z) \rVert \leq L \lVert y - z \rVert$, $\forall ~ y, z \in \real^m$
for some $L > 0.$ 
\end{assumption}

Our analysis of Algorithm~\ref{alg:LU_GT} leads us to find the following convergence rate.
\begin{theorem}[Convergence of LU-GT] \label{thm:convg_lugt} Let Assumptions~\ref{ass:matrix} and \ref{ass:smooth} hold, and let, $T_{\rm o} \in \mathbb{Z}_{\geq 0}$, $\eta>0$ and  \cut{$\alpha>0$} \mbox{such that}
\begin{align*}
    \eta &< \min \left \{1, {(1 - \sqrt{\lambda})}/{(\sqrt{\lambda}(1+T_{\rm o}))}\right\} \\
    \alpha &\leq \min \Biggl\{\sqrt{\frac{(1-\lambda)(1-\theta)}{16 L^2 \lambda}}, \sqrt{\frac{(\bar{\lambda}-\bar{\lambda}^2)\lambda  (1-\theta)}{8L^2\eta^2T^2_{\rm o}}}, \\
    &\qquad \qquad \qquad \qquad  \sqrt[4]{\frac{\lambda (1-\bar{\lambda})^2 (1-\theta)}{32L^4\eta^2T^2_{\rm o}}}, \frac{1}{2L}\Biggl\},
\end{align*}
where $\theta = \lambda(1 + \eta T_{\rm o})^2 < 1$. Then, for any $K\geq 1$, the output~$\x^K$, of Algorithm~\ref{alg:LU_GT} (LU-GT) with \mbox{$\x^0 = \mathbf{1} \otimes x^0$} ($x^0 \in \real^m$) has the following property:
\begin{align}
\frac{1}{K} \hspace{-1mm} \sum_{k=0}^{K-1} \hspace{-1mm} &\big(\lVert \nabla f (\bar{x}^k) \rVert^2 {+} \lVert \overline{\nabla \f} (\x^k) \rVert^2\big) \leq \frac{8\tilde{f}(\bar{x}^0)}{\eta\alpha K}  {+} \frac{4\alpha^2L^2T_{\rm o}\zeta_0}{nK(1-\bar{\lambda})^2},
\end{align}
where $\bar{x}^k  {=} \frac{1}{n}\sum_{i=1}^n x_i^k$, $\overline{\grad \f } (\x^k) {=} \frac{1}{n} \sum_{i=1}^n  \grad f_i (x_i^k)$, $\bar{\lambda} = {(1+\lambda)}/{2}$, $\tilde{f}(\bar{x}^0) = f(\bar{x}^0) - f^*$, $r_k = \lfloor {k}/{T_{\rm o}} \rfloor$, and $\zeta_0 = \lVert \nabla \mathbf{f}(\bar{\mathbf{x}}^0) - \mathbf{1} \otimes \overline{\nabla \mathbf{f}}(\bar{\mathbf{x}}^0) \rVert^2$.
\end{theorem}
{
\begin{rem}\label{rem:1}
If in Theorem~\ref{thm:convg_lugt} we consider a sufficiently well-connected graph where $1 \geq \lambda + \sqrt{\lambda}$ and set $\alpha \propto {(1-\lambda)}/{L}$, and $\eta \propto {1}/{T_{\rm o}}$, then we obtain the convergence rate,
\begin{align*}
    \frac{1}{K} \hspace{-1mm} \sum_{k=0}^{K-1} \bigg( \hspace{-1mm} \lVert \nabla f (\bar{x}^k) \rVert^2 {+} \lVert \overline{\nabla \f} (\x^k) \rVert^2 \hspace{-1mm} \bigg) {\leq }
     O \biggl( \frac{T_{\rm o} \tilde{f}(\bar{x}^0)}{K} {+} \frac{T_{\rm o}\zeta_0}{nK} \hspace{-1mm}  \biggl).
\end{align*}  
\end{rem}
}
{ Theorem~\ref{thm:convg_lugt} implies that LU-GT matches the same communication complexity as the Federated Learning setting (e.g., \cite{karimireddy2020SCAFFOLD}) but allows arbitrary network topologies.}
\section{Preliminaries} \label{sec:prelims}

In this section, we provide intuition for the used proof technique. We show some properties on $\x^k, \y^k$ that will guarantee agents reach a consensus on a first-order stationary point of~\eqref{eq:obj_function}. In addition, we outline a series of transformations performed on~\eqref{eq:GT_rec_1} to simplify the analysis. 

\textbf{Consensus and Optimality:} 
{
Consider the trajectory of a sequence of parameters $\x^k, \y^k$ generated by Algorithm~\ref{alg:LU_GT} converge to a stationary point $(\x, \y)$. Then, from Line~\ref{lst:line:local_rec_1} of Algorithm~\ref{alg:LU_GT} it holds that} $
    \x = \x - \eta \y$, and $\y = \mathbf{0}$.
{ As a result, from Line~\ref{lst:line:comm_rec_1} of Algorithm~\ref{alg:LU_GT},  we have} $\x = \W(\x - \eta \y)$ and $\x = \W\x$. Moreover, initialization, $\y^0 = \alpha \nabla \f (\x^0)$ guarantees that 
\begin{equation} \label{eq: y_avg}
    \bar{y}^{k+1} = \bar{y}^k + \alpha( \overline{\nabla \f}(\x^{k+1}) - \overline{\nabla \f}(\x^{k})) = \alpha \overline{\nabla \f}(\x^{k+1}), 
\end{equation}
where $\bar{y}^k = \frac{1}{n}\sum_{i=1}^n y_i^k$. Hence, at the stationary point $(\x, \y)$, $\alpha \overline{\nabla \f}(\x)=0$. Thus, \mbox{$x_1 = x_2 = ... = x_n$}, i.e, all agents reach the same stationary point.

\textbf{Transformation of Algorithm~\ref{alg:LU_GT}:}
{ We perform a series of transformations on \eqref{eq:GT_rec_1} to simplify the analysis and accurately characterize the behavior of our algorithm. In particular, the trajectory of the average of agent parameters $\bar{x}^k$ is defined to show convergence of $\bar{x}^k$ to a first order stationary point of~\eqref{eq:obj_function}. { Motivated by \cite{ alghunaim2021unified}, the deviation of agent parameters $\x^k$ from the average $\bar{\x}^k \define \bar{x}^k \otimes \one_n$ and the deviation of the gradient tracking variable $\y^k$ from its average \mbox{$\bar{\y}^k \define \bar{y}^k \otimes \one_n$} are considered jointly as one augmented quantity, which simplifies the analysis.}}


Note that the mixing matrix $W$ can be decomposed as 
\begin{align*}
	W=Q \Lambda Q\tran = \begin{bmatrix}
		\frac{1}{\sqrt{n}}  \one & \hat{Q}
	\end{bmatrix} \begin{bmatrix}
		1 & 0 \\
		0 & \hat{\Lambda}
	\end{bmatrix} \begin{bmatrix}
		\frac{1}{\sqrt{n}} \one\tran \vspace{0.6mm} \\  \hat{Q}\tran
	\end{bmatrix},
\end{align*}
where $\hat{\Lambda}=\diag\{\lambda_i\}_{i=2}^n$, $Q$ is a square orthogonal  ($QQ\tran=Q\tran Q=I$), and $\hat{Q}$ is a matrix of size ${n \times (n-1)}$ such that $\hat{Q}\hat{Q}\tran=I_n-\tfrac{1}{n} \one \one\tran$  and  $\one\tran \hat{Q}=0$. From the above, we have
\begin{align*}
	\W&=\Q \bLambda \Q\tran = \begin{bmatrix}
		\frac{1}{\sqrt{n}}   \one \otimes I_m & \hat{\Q}
	\end{bmatrix} \begin{bmatrix}
		I_m & 0 \\
		0 & \hat{\bLambda}
	\end{bmatrix} \begin{bmatrix}
		\frac{1}{\sqrt{n}}  \one\tran \otimes I_m \\ \hat{\Q}\tran
	\end{bmatrix},
\end{align*}
\mbox{where $\hat{\bLambda} \define \hat{\Lambda} \otimes I_m \in \real^{m(n-1)\times m(n-1)}$,  $\Q \in \real^{m n \times m n}$} \mbox{is an orthogonal matrix, and $\hat{\Q}\define \hat{U} \otimes I_m \in \real^{m n \times m(n-1)}$} satisfies:
\begin{align} \label{QQproperties}
	\hat{\Q}\tran \hat{\Q}{=}\I, \ \hat{\Q}\hat{\Q}\tran{=}\I{-}\tfrac{1}{n} \one \one\tran \otimes I_m, \ (\one\tran \otimes I_m) \hat{\Q}{=}0.
\end{align}
Using the notation defined in \eqref{def:W_k}, it then follows that
\begin{equation} \label{def:Lambda_k}
    \bLambda_k \triangleq \begin{cases} \bLambda & \text{when } k \;\mathrm{mod}\; (T_{\rm o}-1) = 0, \\ \I & \text{otherwise}. \end{cases}
\end{equation}



{
Equation~\eqref{QQproperties} directly leads to
\begin{align*}
    \lVert \hat{\Q}\tran \x \rVert^2 &= \x\tran \hat{\Q}\hat{\Q}\tran\hat{\Q}\hat{\Q}\tran \x = \lVert \hat{\Q} \hat{\Q}\tran \x \rVert^2 = \lVert \x - \bar{\x} \rVert^2 \\
    \lVert \hat{\Q}\tran \y \rVert^2 &= \y\tran \hat{\Q}\hat{\Q}\tran\hat{\Q}\hat{\Q}\tran \y = \lVert \hat{\Q} \hat{\Q}\tran \y \rVert^2 = \lVert \y - \bar{\y} \rVert^2.
\end{align*}

In addition, we know that $\Q \define \begin{bmatrix} \frac{1}{\sqrt{n}}   \one \otimes I_m & \hat{\Q} \end{bmatrix}$. To recover the average $\bar{x}$ from the augmented vector $\x$, the following operation can be performed $
(\frac{1}{n}\one\tran \otimes I_m) \x = \bar{x}$.
}
Hence, we multiply (\ref{eq:GT_rec_1}) by $\Q \tran$ and simplify to get
\begin{IEEEeqnarray}{rCl}
    \Q{\tran} \x^{k+1} &=& \bLambda_{k} \Q{\tran} (\x^k - \eta \y^k) \nonumber\\ 
    \Q{\tran} \y^{k+1} &=& {\bLambda}_{k} \Q\tran \y^k +  \alpha {\bLambda}_{k} \Q\tran (\nabla \f (\x^{k+1}) - \nabla \f (\x^k)) \nonumber\\ 
    \begin{bmatrix} \bar{x}^{k+1} \\ \hat{\Q}\tran \x^{k+1} \end{bmatrix} &=& \bLambda_{k} \left( \begin{bmatrix} \bar{x}^{k} \\ \hat{\Q}\tran \x^{k} \end{bmatrix} - \begin{bmatrix}  \eta \alpha \overline{\nabla \f}(\x^{k}) \\ \eta \hat{\Q}\tran \y^k \end{bmatrix} \right) \nonumber \\
    \begin{bmatrix} \overline{\nabla \f}(\x^{k+1}) \\ \hat{\Q}\tran \y^{k+1} \end{bmatrix} &{=}&  {\bLambda}_{k} \begin{bmatrix} \overline{\nabla \f}(\x^k)\\ \hat{\Q}\tran \y^k \end{bmatrix}\nonumber\\
    &+&\alpha {\bLambda}_{k} \begin{bmatrix} \overline{\nabla \f}(\x^{k+1)}) - \overline{\nabla \f}(\x^{k}) \\ \hat{\Q}\tran (\nabla \f (\x^{k+1}) - \nabla \f (\x^k))  \end{bmatrix}. \nonumber
\end{IEEEeqnarray}

Using the structure of ${\bLambda}_{k}$, we then have
\begin{IEEEeqnarray}{rCl}\label{eq:final_transformed_recursion}
    \bar{x}^{k+1} &{=}& \bar{x}^k - \eta \alpha \overline{\nabla \f}(\x^k) \IEEEyessubnumber \label{eq:avg_x_desc}\\
    \hat{\Q}\tran \x^{k+1} &{=}& \hat{\bLambda}_k \hat{\Q}\tran (\x^k - \eta \y^k) \IEEEyessubnumber \\
    \hat{\Q}\tran \y^{k{+}1} &{=}& \hat{\bLambda}_k \hat{\Q}\tran \y^k {+}\alpha \hat{\bLambda}_k \hat{\Q}\tran (\nabla \f(\x^{k{+}1}) {-} \nabla \f(\x^{k})). \IEEEyessubnumber
\end{IEEEeqnarray}
\cut{ Equation~\eqref{eq:avg_x_desc} shows that the average of agent parameters, $\bar{x}$, is updated by performing a gradient descent step using the global gradient evaluated at the past average gradients. Then, $\bar{x}$ will converge to a stationary point in the limit. Therefore, if agents reach a consensus, this consensus will be a stationary point of~\eqref{eq:obj_function}.}  
{
We then convert~\eqref{eq:final_transformed_recursion} into matrix notation 
\begin{IEEEeqnarray}{rCl}\label{eq:final_transformed_recursion_matrix_x_delta}
    &&\bar{x}^{k+1} = \bar{x}^k - \eta \alpha \overline{\nabla \f}(\x^k) \IEEEyessubnumber \\
    &&\begin{bmatrix}
        \hat{\Q}\tran \x^{k+1} \\ \hat{\Q}\tran \y^{k+1}  
    \end{bmatrix} = \begin{bmatrix}
        \hat{\bLambda}_k & -\eta \hat{\bLambda}_k \\
        \mathbf{0} & \hat{\bLambda}_k
    \end{bmatrix} \begin{bmatrix}
        \hat{\Q}\tran \x^{k} \\ \hat{\Q}\tran \y^{k} \nonumber  
    \end{bmatrix} \nonumber \\
    &&\quad \quad \quad + \alpha \begin{bmatrix}
        \mathbf{0} \\
        \hat{\bLambda}_k \hat{\Q}\tran (\nabla \f(\x^{k+1}) - \nabla \f(\x^{k}))
    \end{bmatrix}. \IEEEyessubnumber
\end{IEEEeqnarray}

}

By iterating~\eqref{eq:final_transformed_recursion_matrix_x_delta} up to $r_k T_{\rm o}$ where $r_k = \lfloor k/T_{\rm o} \rfloor$ it follows that
{
\eqref{eq:GT_rec_1} can be rewritten as 
\begin{IEEEeqnarray}{rCl}\label{eq:final_transformation}
    \bar{x}^{k+1} &=& \bar{x}^k - \eta \alpha \overline{\nabla \f}(\x^k) \IEEEyessubnumber\\
    \bPhi^{k+1} &=& \left( \prod_{l=k}^{r_kT_{\rm o}} \mathbf{G}_l
     \right) \bPhi^{r_kT_{\rm o}} \nonumber \\
     &{+}& \alpha \h^{k+1} {+} \alpha \sum_{t= r_k T_{\rm o}}^{k-1} \left( \prod_{l=k-1}^{t} \mathbf{G}_l \right) \h^{t{+}1} \IEEEyessubnumber
\end{IEEEeqnarray}
where $\bPhi^k \define \begin{bmatrix} \hat{\Q}\tran \x^k \\ \hat{\Q}\tran \y^k
    \end{bmatrix} ,\quad \G_k \define \begin{bmatrix}
        \hat{\bLambda}_k & -\eta\hat{\bLambda}_k \\
        \mathbf{0} & \hat{\bLambda}_k
    \end{bmatrix}$, and
\begin{align*}
    \quad \h^{k+1} &\define \begin{bmatrix}
        \mathbf{0} \\
        \hat{\bLambda}_k \hat{\Q}\tran (\nabla \f(\x^{k+1}) - \nabla \f(\x^{k}))
    \end{bmatrix}. 
\end{align*} 
In the next Section, we analyze and bound the trajectory of the augmented consensus quantity $\lVert \bPhi^k \rVert$ necessary to establish the convergence of Algorithm~\ref{alg:LU_GT}. 
}

\section{Analysis on Convergence of Algorithm~\ref{alg:LU_GT}} \label{sec:analysis}


{ \cut{In this section we prove our main result in Theorem~\ref{thm:convg_lugt}. We start by introducing a series of technical lemmas that will help us build the desired result.} Lemma~\ref{lem:prod_matrix_local} and Lemma~\ref{lem:prod_matrix_one_round} quantify the effect of local steps and the mixing matrix $W$ on the augmented consensus quantity. Lemma~\ref{lem:param_fluc} provides a bound of the deviation of the parameter $\x^k$ between iterations. This is needed in Lemma~\ref{lem:h_ineq} to bound the quantity $\h^k$ used in the bound on the augmented consensus quantity.}

\begin{lemma} \label{lem:prod_matrix_local} 
For iterates $t$, and $k$ of \mbox{Algorithm~\ref{alg:LU_GT}} where \mbox{$r_{k}T_{\rm o} < t < k-1$ and $k-1, t \notin \tau$, the following matrix} inequality holds
\begin{equation}
    \left \lVert \prod_{l=k-1}^{t} \G_l \right \rVert  \leq 1 + \eta (k-1-t) <  1 + \eta T_{\rm o}.
\end{equation} 
\end{lemma}
\proof
\begin{align*}
   & \left \lVert \prod_{l=k-1}^{t} \G_l \right \rVert = \left \lVert \begin{bmatrix}
            \I & -\eta(k-1-t)\I \\
            \mathbf{0} & \I
        \end{bmatrix}  \right \rVert \\
        &= \left \lVert \I {+} \begin{bmatrix}
            \mathbf{0} & {-}\eta(k{-}1{-}t)\I \\
            \mathbf{0} & \mathbf{0}
        \end{bmatrix} \right \rVert \leq 1 {+} \eta (k {-} 1 {-} t) \leq 1 {+} \eta T_{\rm o}.
\end{align*}
{
The first equality follows from multiplying \mbox{$\G_l$} from \mbox{$l=t$} to \mbox{$l=k-1$}. The second equality follows from directly decomposing the result matrix product as a sum. The final step uses the sub-additive property of matrix norms.}
\endproof
\begin{lemma} \label{lem:prod_matrix_one_round} Suppose that Assumption~\ref{ass:matrix} holds. For an iterate \mbox{$k$ of Algorithm~\ref{alg:LU_GT} where \mbox{$r_{k}T_{\rm o} < k$} and $k \notin \tau$, the} following matrix inequality holds
\begin{equation}
    \left \lVert \prod_{l=k}^{r_k T_{\rm o}} \G_l  \right \rVert \leq \lambda(1 + \eta T_{\rm o} ).
\end{equation}
\end{lemma}
\proof
    \begin{align*}
       \prod_{l=k}^{r_k T_{\rm o}} \G_l  \define \C &=  \begin{bmatrix}
            \hat{\bLambda} & -\eta \hat{\bLambda} \\
            \mathbf{0} & \hat{\bLambda} 
        \end{bmatrix}
        \begin{bmatrix}
            \I & -\eta(k-r_k T_{\rm o})\I \\
            \mathbf{0} & \I
        \end{bmatrix} \\ &= \begin{bmatrix}
            \hat{\bLambda} & -\eta(k - r_k T_{\rm o}) \hat{\bLambda}- \eta\hat{\bLambda} \\
            \mathbf{0} & \hat{\bLambda}
        \end{bmatrix}.
    \end{align*}
\par
This result directly follows from multiplying \mbox{$\G_l$} from \mbox{$l=r_k T_{\rm o}$} to \mbox{$l=k$}. 
\par
There exists a coordinate of transformation matrix $\R$ such that $\R\tran \C \R = \text{blkdiag} \{\C\}_{i=2}^n $,
where
\begin{align*}
    \C_i &= \begin{bmatrix}
        \lambda_i & -\eta (k - r_k T_{\rm o}) \lambda_i - \eta \lambda_i \\
        0 & \lambda_i
    \end{bmatrix} \\
    &= \lambda_i \left ( I + \begin{bmatrix}
        0 & -\eta (k - r_k T_{\rm o}) - \eta  \\
        0 & 0
    \end{bmatrix}\right).
\end{align*}
To get the above result, we factored out $\lambda_i$ and decomposed the matrix as a sum of matrices. Hence,
\begin{align*}
\lVert \C_i \rVert &\leq \lambda_i(1 + \eta (k - r_k T_{\rm o})+ \eta) \\
\lVert \C \rVert &\leq \lambda(1 + \eta (k - r_k T_{\rm o})+ \eta) \leq \lambda(1 + \eta(T_{\rm o})).
\end{align*}
Here we first used the sub-additive property of matrix norms. Then, we took advantage of the block-diagonal structure of \mbox{$\R\tran \C \R$} and the fact that $\lVert \R \rVert = 1$.
\endproof
\begin{lemma} \label{lem:param_fluc} \cut{\mbox{Let $0 < \eta, \alpha < 1$} and $k \geq 0$. Then, an iterate $\x^k$ of Algorithm~\ref{alg:LU_GT} has the following property:
\begin{equation*}
    \lVert \x^{k+1} - \x^{k} \rVert^2{=}\begin{cases}
       4 \lVert \bPhi^k \rVert^2 {+} 4n\eta^2\alpha^2 \lVert \overline{\nabla \f} (\x^k) \rVert^2 & k \in \tau, \\
        4 \eta^2 \lVert \bPhi^k \rVert^2 {+} 4n \eta^2 \alpha^2 \lVert \overline{\nabla \f} (\x^k) \rVert^2  & \text{else}.
    \end{cases}
\end{equation*}
where $n$ is the number of agents.}
\end{lemma}

\proof
{ Depending on \mbox{$k$}, we have two possibilities} 
\begin{equation*}
    \lVert \x^{k+1} - \x^{k} \rVert^2 = \begin{cases}
        \lVert (\W - \I)\x^{k} - \eta\W \y^{k} \rVert^2 & \text{when } k \in \tau, \\
        \lVert \eta \y^k \rVert^2 & \text{otherwise}.
    \end{cases}
\end{equation*}
We start by bounding the first case:
\begin{align*}
    &\lVert (\W - \I)\x^k - \eta \W \y^k \rVert^2 \\
    &= \lVert (\W - \I)(\x^k - (\mathbf{1} \otimes \bar{x}^k)) - \eta \W \y^k\rVert^2 \\
    &\leq 4 \lVert \x^k - \bar{\x}^k \rVert^2 + 2 \lVert \eta \y^k \rVert^2 \\
    &\leq 4 \lVert \x^k - \bar{\x}^k \rVert^2 + 4 \eta^2 \lVert \y^k - \bar{\y}^k \rVert^2  + 4 \eta^2 \lVert \bar{\y}^k \rVert^2
\end{align*}
{ The first equality adds and subtracts $\mathbf{1} \otimes \bar{x}^k$ inside the norm. We take advantage of the fact that $\W (\mathbf{1} \otimes \bar{x}^k)) = \mathbf{1} \otimes \bar{x}^k$. In the first inequality, we use \mbox{$\lVert a + b \rVert^2 \leq 2\lVert a \rVert^2 + 2\lVert b \rVert^2$} twice and then use Assumption~\ref{ass:matrix} to upper bound the spectral norm of $\W$ by $1$. In the final inequality, we use \mbox{$\lVert a + b \rVert^2 \leq 2\lVert a \rVert^2 + 2\lVert b \rVert^2$}.
} Using \eqref{eq: y_avg}, we have
\begin{align*}
\lVert (\W - \I) \x^k - \W \y^k\rVert^2 &\leq 4 \lVert \x^k - \bar{\x}^k \rVert^2 + 4 \eta^2 \lVert \y^k - \bar{\y}^k \rVert^2 \\
&\quad + 4n\eta^2\alpha^2 \lVert \overline{\nabla \f} (\x^k) \rVert^2.
\end{align*}
Using the properties in \eqref{QQproperties} the following upper bound on the consensus error holds
\begin{align*}
    \lVert \bPhi^k \rVert^2 &= \left \lVert \begin{bmatrix}
        \hat{\Q}\tran \x^k \\
        \hat{\Q}\tran \y^k
    \end{bmatrix} \right \rVert^2 = \lVert \x^k - \bar{\x}^k \rVert^2 + \lVert \y^k - \bar{\y}^k \rVert^2.
\end{align*}
Since $0 < \eta < 1$ it follows that
\begin{align} \label{eq:x_y_consensus_bound}
    \lVert \x^k - \bar{\x}^k \rVert^2 + \eta^2 \lVert \y^k - \bar{\y}^k \rVert^2 \leq \lVert \bPhi^k \rVert^2.
\end{align}
Hence, 
\begin{equation*}
\lVert (\W - \I) \x^k - \W \y^k\rVert^2 \leq 4 \lVert \bPhi^k \rVert^2 + 4n\eta^2\alpha^2 \lVert \overline{\nabla \f} (\x^k) \rVert^2.
\end{equation*}
We now bound the second case
\begin{align*}
    \lVert \eta \y^k \rVert^2 &= \eta^2 \lVert \y^k - \bar{\y}^k + \bar{\y}^k \rVert^2\\
    &\leq 2 \eta^2 \lVert \hat{\Q}\tran \y^k \rVert^2 + 2n\eta^2\alpha^2 \lVert \overline{\nabla \f} (\x^k) \rVert^2 \\
    &\leq 4 \eta^2 \lVert \bPhi^k \rVert^2 + 4n \eta^2 \alpha^2 \lVert \overline{\nabla \f} (\x^k) \rVert^2.
\end{align*}
{ In the first inequality, we apply \mbox{$\lVert a + b \rVert^2 \leq 2\lVert a \rVert^2 + 2\lVert b \rVert^2$} and use~\eqref{QQproperties} on the term $\lVert \bar{\y}^k - \bar{\y}^k \rVert$. Then, we use \eqref{eq: y_avg} on the term $\bar{\y}^k$. In the final inequality, we use \eqref{eq:x_y_consensus_bound}}.
\endproof
\begin{lemma} \label{lem:h_ineq} Let Assumptions~\ref{ass:matrix} and~\ref{ass:smooth} hold. For an \mbox{iteration $k \notin \tau$, step-size $\alpha>0$, smoothness parameter $L$} {defined in Assumption~\ref{ass:smooth}, constant $1 > \eta> 0$, and number of} local iterations $T_{\rm o}$, the following inequality holds  
\begin{align}
&\lVert \h^{k+1} \rVert^2 + \left \lVert \sum_{t= r_k T_{\rm o}}^{k-1} \left( \prod_{l = k-1}^{t} \G_l \right) \h^{t+1} \right \rVert^2 \\ \nonumber 
&\leq 8L^2 \eta^2 T_{\rm o} (1+\eta T_{\rm o})^2 \sum_{t = r_k T_{\rm o}+1}^{k} (\lVert \bPhi^t \rVert^2 {+} n \alpha^2 \lVert \overline{\nabla \f} (\x^t) \rVert^2)\\
&+ 8L^2\lambda^2(1+\eta T_{\rm o}) (\lVert \bPhi^{r_k T_{\rm o}} \rVert^2 {+} n\eta^2\alpha^2 \lVert \overline{\nabla \f} (\x^{r_k T_{\rm o}}) \rVert^2). 
\end{align}
\end{lemma}
\proof
First, we use \mbox{$\lVert a + b \rVert^2 \leq 2\lVert a \rVert^2 + 2\lVert b \rVert^2$} to obtain,
\begin{align*}
    &\lVert \h^{k+1} \rVert^2 + \left \lVert \sum_{t= r_k T_{\rm o}}^{k-1} \left( \prod_{l = k-1}^{t} \G_l \right) \h^{t+1} \right \rVert^2 \\
    &\leq \lVert \h^{k+1} \rVert^2 + 2 \left \lVert  \left( \prod_{l = k-1}^{r_k T_{\rm o}} \G_l \right) \h^{r_k T_{\rm o}+1} \right \rVert^2 \\
    &+ 2\left \lVert \sum_{t= r_k T_{\rm o}+1}^{k-1} \left( \prod_{l = k-1}^{t} \G_l \right) \h^{t+1} \right \rVert^2\\
    &\leq L^2 \lVert \x^{k+1}{-}\x^{k} \rVert^2{+}2L^2\lambda^2(1+\eta T_{\rm o}) \lVert \x^{r_k T_{\rm o}+1} - \x^{r_k T_{\rm o}}\rVert^2 \\
    &+ 2 L^2 T_{\rm o} (1+\eta T_{\rm o})^2 \sum_{t = r_k T_{\rm o}+1}^{k-1} \left \lVert \x^{t+1}{-}\x^{t} \right \rVert^2 \\
    &\leq 4L^2\eta^2 (\lVert \bPhi^k \rVert^2 {+} n \alpha^2 \lVert \overline{\nabla \f} (\x^k) \rVert^2) \\
    &+ 8L^2\lambda^2(1+\eta T_{\rm o}) (\lVert \bPhi^{r_k T_{\rm o}} \rVert^2 {+} n\eta^2\alpha^2 \lVert \overline{\nabla \f} (\x^{r_k T_{\rm o}}) \rVert^2) \\
    &+ 8L^2 \eta^2 T_{\rm o} (1+\eta T_{\rm o})^2 \sum_{t = r_k T_{\rm o}+1}^{k-1} (\lVert \bPhi^t \rVert^2 {+} n \alpha^2 \lVert \overline{\nabla \f} (\x^t) \rVert^2) \\
    &\leq 8L^2\lambda^2(1+\eta T_{\rm o}) (\lVert \bPhi^{r_k T_{\rm o}} \rVert^2 {+} n\eta^2\alpha^2 \lVert \overline{\nabla \f} (\x^{r_k T_{\rm o}}) \rVert^2) \\
    &+ 8L^2 \eta^2 T_{\rm o} (1+\eta T_{\rm o})^2 \sum_{t = r_k T_{\rm o}+1}^{k} (\lVert \bPhi^t \rVert^2 {+} n \alpha^2 \lVert \overline{\nabla \f} (\x^t) \rVert^2).
\end{align*}
{
In the second inequality, we used Lemma~\ref{lem:prod_matrix_local}, Lemma~\ref{lem:prod_matrix_one_round}, and Assumption~\ref{ass:smooth}. In the third inequality, we used Lemma~\ref{lem:param_fluc}.} In the fourth inequality, we group similar terms.
\endproof
Next, we find a bound on the consensus inequality to later use in the descent inequality. Note that we define $\sum_{t=r_k T_{\rm o}}^k (\cdot)$ as zero if $r_k T_{\rm o} > k - 1$.

\begin{lemma}[Consensus Inequality] \label{lem:consen_ineq} 
Let Assumptions~\ref{ass:matrix} and~\ref{ass:smooth} and
\begin{align*}
 \eta &< \min \left \{1, {(1 - \sqrt{\lambda})}/{(\sqrt{\lambda}(1+T_{\rm o}))}\right\},\\
    \alpha &\leq \min \Biggl\{ \sqrt{\frac{\lambda(1-\lambda)(1-\theta)}{16L^2T_{\rm o}}}, \sqrt{\frac{(\bar{\lambda}-\bar{\lambda}^2)\lambda  (1-\theta)}{8L^2\eta^2T^2_{\rm o}}}  \Biggl\},
\end{align*}
hold. Then, the output of Algorithm \eqref{alg:LU_GT} satisfies the following inequality 
\begin{align}
&\frac{1}{K}  \sum_{k=0}^{K-1} \lVert \bPhi^k \rVert^2 \leq \frac{(1-\bar{\lambda}) (\frac{1}{K} \sum_{k=0}^{K-1}\bar{\lambda}^{r_{k-1}+1})}{1-\bar{\lambda}-e_1T_{\rm o}} \lVert \bPhi^0 \rVert^2 \nonumber \\
      &+ \left( \frac{e_2 T_{\rm o}}{K(1 - \bar{\lambda} - e_1 T_{\rm o})} \right) \sum_{k=0}^{K-1} \left ( \lVert \overline{\nabla \mathbf{f}}(\mathbf{x}^k) \rVert^2 + \lVert \nabla f (\bar{x}^k) \rVert^2 \right),
\end{align}
\mbox{where $e_1 \define \frac{8L^2\eta^2\alpha^2 T_{\rm o} (1+\eta T_{\rm o})^2}{(1-\theta)},e_2 \define \frac{8nL^2\eta^2\alpha^4 T_{\rm o} (1+\eta T_{\rm o})^2}{(1 - \theta)}$}, and $r_k \define \lfloor k/T_{\rm o} \rfloor$.
\end{lemma}
\proof
We take the norm of (\ref{eq:final_transformation}) and apply Jensen's inequality for any $0 < \theta < 1$.
\begin{align*}
&\lVert \bPhi^{k+1}\rVert^2 \leq \frac{1}{\theta} \left \lVert \left( \prod_{l=k}^{r_kT_{\rm o}} \mathbf{G}_l
     \right) \bPhi^{r_kT_{\rm o}} \right \rVert^2 \\
     &+ \frac{2\alpha^2}{(1-\theta)}  \Biggl( \lVert \h^{k+1} \rVert^2 + \left \lVert \sum_{t= r_k T_{\rm o}}^{k-1} \left( \prod_{l=k-1}^{t} \mathbf{G}_l \right) \h^{t+1} \right \rVert^2 \Biggl)\\
     &\quad \\
     &\leq \frac{\lambda^2 (1 + \eta T_{\rm o})^2}{\theta}  \left \lVert \bPhi^{r_kT_{\rm o}} \right \rVert^2 \\
     &+ \frac{8L^2 \alpha^2 (1+\eta T_{\rm o})^2}{ (1-\theta)} \left( \sum_{t=r_k T_{\rm o}+1}^{k} T_{\rm o} \eta^2\lVert \bPhi^t \rVert^2 + \lambda^2 
     \lVert \bPhi^{r_k T_{\rm o}} \rVert^2 \right)\\
     &+\frac{8nL^2 \eta^2 \alpha^4 (1+\eta T_{\rm o})^2}{(1-\theta)}\Biggl(\sum_{t=r_k T_{\rm o}+1}^{k} T_{\rm o} \lVert \overline{\nabla \f} (\x^k) \rVert^2 \\
     &+ \lambda^2 \lVert \overline{\nabla \f} (\x^{r_k T_{\rm o}}) \rVert^2) \Biggl).
\end{align*}
In the second inequality, we applied the results from Lemma~\ref{lem:prod_matrix_one_round} and Lemma~\ref{lem:h_ineq}. Set $\theta = \lambda(1 + \eta T_{\rm o})^2 < 1 \Rightarrow \eta < \frac{1 - \sqrt{\lambda}}{\sqrt{\lambda}(1+T_{\rm o})}$. Moreover, define $
    e_1 \define \frac{8L^2\eta^2\alpha^2 T_{\rm o} (1+\eta T_{\rm o})^2}{(1-\theta)}$, and $e_2 \define \frac{8nL^2\eta^2\alpha^4 T_{\rm o} (1+\eta T_{\rm o})^2}{(1 - \theta)}$, which allows us to obtain the following
\begin{align*}
\lVert \bPhi^{k+1} \rVert^2 &\leq (\lambda+\frac{\lambda^2 e_1}{T_{\rm o}\eta^2}) \left \lVert \bPhi^{r_kT_{\rm o}} \right \rVert^2 + e_1 \sum_{t=r_k T_{\rm o}+1}^{k} \lVert \bPhi^t \rVert^2 + \\
&+ e_2 \left( \sum_{t=r_k T_{\rm o} +1 }^{k} \lVert \overline{\nabla \f} (\x^t) \rVert^2+ \frac{\lambda^2}{T_{\rm o}}\lVert \overline{\nabla \f} (\x^{r_t T_{\rm o}}) \rVert^2 \right).
\end{align*}
Choose $\alpha$ such that
\begin{equation*}
    \lambda + \frac{\lambda^2 e_1}{T_{\rm o}\eta^2} \leq \frac{1+\lambda}{2} \Rightarrow \alpha \leq \sqrt{\frac{(1-\lambda)(1-\theta)}{16 L^2 \lambda}}.
\end{equation*}
Defining $\bar{\lambda} = (1+\lambda)/2$ and observing that $\frac{\lambda^2}{T_{\rm o}} < 1$, we have
\begin{equation} \label{eq:ineq_cons}
    \lVert \bPhi^{k+1} \rVert^2 {\leq} \bar{\lambda} \lVert \mathbf{\bPhi}^{r_k T_{\rm o}} \rVert^2 + e_1 \hspace{-4mm}\sum_{t = r_k T_{\rm o}+1}^k \hspace{-2mm} \lVert \mathbf{\bPhi}^t \rVert^2 + e_2 \hspace{-2mm}\sum_{t=r_k T_{\rm o}}^k \hspace{-2mm} \lVert \overline{\nabla \f} (\x^t) \rVert^2.
\end{equation}
When $k = r_k T_{\rm o}-1$, we have 
$$\lVert \bPhi^{r_k T_{\rm o}} \rVert^2 {\leq} \bar{\lambda} \lVert \bPhi^{(r_{k-1}) T_{\rm o}} \rVert^2 {+} e_1 \hspace{-0.5cm} \sum_{t = (r_{k}-1) T_{\rm o}}^{r_k T_{\rm o} - 1}  \hspace{-0.5cm}  \lVert \bPhi^t \rVert^2 + e_2   \hspace{-0.5cm} \sum_{t=(r_{k}-1) T_{\rm o}}^{r_k T_{\rm o} - 1} \hspace{-0.5cm} \lVert \overline{\nabla \f} (\x^t) \rVert^2.$$  
Substitute the above into \eqref{eq:ineq_cons} and iterate to find
\begin{align*}
    \lVert \bPhi^{k+1} \rVert^2 &\leq \bar{\lambda}^{r_k} \lVert \bPhi^0 \rVert^2  + e_1 \Biggl (\sum_{t=r_k T_{\rm o}}^k \lVert \bPhi^{t} \rVert^2 + \nonumber \\ 
    & \bar{\lambda} \sum_{t=(r_k -1)T_{\rm o}}^{r_k T_{\rm o} - 1} \lVert \bPhi^{t}\rVert^2 + \cdots +  \bar{\lambda}^{r_k} \sum_{t=0}^{T_{\rm o} - 1} \lVert \bPhi^{t}\rVert^2 \Biggl) \\ 
    &+ e_2  \Biggl(\sum_{t = r_kT_{\rm o}}^k \lVert \overline{\nabla \f}(\x^t) \rVert^2  
    + \bar{\lambda} \sum_{t=(r_k - 1)T_{\rm o}}^{r_k T_{\rm o}-1} \lVert \overline{\nabla \f}(\x^t) \rVert^2 \\
    &+ \cdots \bar{\lambda}^{r_k} \sum_{t=0}^{T_{\rm o} - 1} \lVert \overline{\nabla \f} (\x^k) \rVert^2 \Biggl).
\end{align*}
Recall that $r_k = \lfloor \frac{k}{T_{\rm o}} \rfloor$. Thus, we introduce the notation 
\begin{equation*}
	\bar{\lambda}^{(k, t)} \triangleq 
	\begin{cases}
        0 & t \leq -1 \\
		1 & r_k T_{\rm o} \leq t \leq k \\
		\bar{\lambda} & (r_k-1) T_{\rm o} \leq t \leq r_k T_{\rm o}-1 \\
		\bar{\lambda}^2 & (r_k-2) T_{\rm o} \leq t \leq (r_k-1) T_{\rm o}-1 \\
		\vdots & \vdots \\
		\bar{\lambda}^{r_k} & 0 \leq t \leq T_{\rm o} -1.
	\end{cases}
\end{equation*}
We can then describe the previous bound more compactly as
\begin{align*}
    \lVert \bPhi^{k} \rVert^2 &\leq \bar{\lambda}^{r_{k-1} + 1}\lVert \bPhi^0 \rVert^2 \\
    &+ e_1 \sum_{t=0}^{k-1} \bar{\lambda}^{(k-1, t)} \lVert \bPhi^t \rVert^2 + e_2 \sum_{t=0}^{k-1} \bar{\lambda}^{(k-1, t)} \lVert \overline{\nabla \f} (\x^t) \rVert^2.
\end{align*}
when setting $k+1$ as $k$. Then, we average over $k = 0,...,K-1$ and upper bound the result as follows
\begin{align*}
 &\frac{1}{K} \sum_{k=0}^{K-1} \lVert \bPhi^k \rVert^2 \leq   \\
 &\frac{1}{K} \sum_{k=0}^{K-1} \bar{\lambda}^{r_{k-1} + 1}\lVert \bPhi^0 \rVert^2 + \frac{e_1}{K} \sum_{k=0}^{K-1} \sum_{t=0}^{k-1} \bar{\lambda}^{(k-1, t)} \lVert \bPhi^t \rVert^2 \\
 &+ \frac{e_2}{K} \sum_{k=0}^{K-1} \sum_{t=0}^{k-1} \bar{\lambda}^{(k-1, t)} \lVert \overline{\nabla \f} (\x^t) \rVert^2\\
 &= \frac{1}{K} \sum_{k=0}^{K-1} \bar{\lambda}^{r_{k-1} + 1}\lVert \bPhi^0 \rVert^2 + \frac{e_1}{K}\sum_{t=0}^{K-1} \sum_{k=t}^{K-1} \bar{\lambda}^{(k-1, t)} \lVert \bPhi^t \rVert^2  \\
 &+ \frac{e_2}{K} \sum_{t=0}^{K-1} \sum_{k=t}^{K-1} \bar{\lambda}^{(k-1, t)} \lVert \overline{\nabla \f} (\x^t) \rVert^2 \\
 &= \frac{1}{K} \sum_{k=0}^{K-1} \bar{\lambda}^{r_{k-1} + 1}\lVert \bPhi^0 \rVert^2+ \frac{e_1}{K}\sum_{t=0}^{K-1}  \lVert \bPhi^t \rVert^2 \sum_{k=t}^{K-1} \bar{\lambda}^{(k-1, t)}\\
 &+\frac{e_2}{K} \sum_{t=0}^{K-1} \lVert \overline{\nabla \f} (\x^t) \rVert^2 \sum_{k=t}^{K-1} \bar{\lambda}^{(k-1, t)} \\
 &\leq \frac{1}{K} \sum_{k=0}^{K-1} \bar{\lambda}^{r_{k-1} + 1}\lVert \bPhi^0 \rVert^2 + \frac{e_1 T_{\rm o}}{K(1-\bar{\lambda})}\sum_{k=0}^{K-1}  \lVert \bPhi^k \rVert^2\\
 &+ \frac{e_2 T_{\rm o}}{K (1-\bar{\lambda})} \sum_{k=0}^{K-1} \lVert \overline{\nabla \f} (\x^k) \rVert^2.
\end{align*}
{ In the first equality, we rearrange the order of the summation. In the second equality, we rearrange the terms in the summation based on their index. In the second inequality, we upper bound \mbox{$\sum_{k=0}^{K-1}\sum_{k=t}^{K-1} \bar{\lambda}^{(k-1, t)}$} with $T_{\rm o}/(1-\bar{\lambda})$ and change the indexing from $t$ to $k$ afterwards.}  
Therefore,
\begin{align*}
    \left(1 - \frac{e_1T_{\rm o}}{1 - \bar{\lambda}} \right) \frac{1}{K} &\sum_{k=0}^{K-1} \lVert \bPhi^k \rVert^2 \leq \frac{1}{K} \sum_{k=0}^{K-1} \bar{\lambda}^{r_{k-1} + 1} \lVert \bPhi^{0} \rVert^2 \\
    &+ \left ( \frac{e_2 T_{\rm o}}{K (1 - \bar{\lambda})} \right) \sum_{k=0}^{K-1} \lVert \overline{\nabla \mathbf{f}} (\x^k) \rVert^2 ,
\end{align*}
and with further simplification, we have
\begin{align}
    \left( \frac{1 - \bar{\lambda} - e_1T_{\rm o}}{1 - \bar{\lambda}} \right) &\frac{1}{K} \sum_{k=0}^{K-1} \lVert \bPhi^k \rVert^2 \leq \frac{1}{K} \sum_{k=0}^{K-1} \bar{\lambda}^{r_{k-1} + 1} \lVert \bPhi^{0} \rVert^2 \\
    &+ \left ( \frac{e_2 T_{\rm o}}{K (1 - \bar{\lambda})} \right) \sum_{k=0}^{K-1} \lVert \overline{\nabla \mathbf{f}} (\x^k) \rVert^2. 
\end{align}

By imposing the following assumption on $\alpha$
\begin{equation*}
     \frac{1 - \bar{\lambda} - e_1T_{\rm o}}{1 - \bar{\lambda}}  \geq 1-\bar{\lambda} \Rightarrow \alpha \leq \sqrt{\frac{(\bar{\lambda}-\bar{\lambda}^2)\lambda  (1-\theta)}{8L^2\eta^2T^2_{\rm o}}},
\end{equation*}
it follows that
\begin{align*}
		&\frac{1}{K} \sum_{k=0}^{K-1} \lVert \bPhi^k \rVert^2 \leq \frac{(1-\bar{\lambda}) (\frac{1}{K} \sum_{k=0}^{K-1}\bar{\lambda}^{r_{k-1}+1})}{1-\bar{\lambda}-e_1T_{\rm o}} \lVert \bPhi^0 \rVert^2 \\
        &+ \left( \frac{e_2 T_{\rm o}}{K(1 - \bar{\lambda} - e_1 T_{\rm o})} \right) \sum_{k=0}^{K-1} \lVert \overline{\nabla \mathbf{f}}(\mathbf{x}^k) \rVert^2 \\
      &\leq \frac{(1-\bar{\lambda}) (\frac{1}{K} \sum_{k=0}^{K-1}\bar{\lambda}^{r_{k-1}+1})}{1-\bar{\lambda}-e_1T_{\rm o}} \lVert \bPhi^0 \rVert^2 \\
      &+ \left( \frac{e_2 T_{\rm o}}{K(1 - \bar{\lambda} - e_1 T_{\rm o})} \right) \sum_{k=0}^{K-1} \left ( \lVert \overline{\nabla \mathbf{f}}(\mathbf{x}^k) \rVert^2 + \lVert \nabla f (\bar{x}^k) \rVert^2 \right).
\end{align*}
\endproof
We are now ready to state the proof of Theorem~\ref{thm:convg_lugt}.
\begin{proof}[Proof of Theorem~\ref{thm:convg_lugt}]
Following similar arguments as in \cite[Lemma 3]{alghunaim2021unified}, we have the following inequality 
\begin{align} \label{eq:descent_inequality}
    f(\bar{x}^{k+1}) &\leq f(\bar{x}^k) {-} \frac{\eta\alpha}{2} \lVert \nabla f (\bar{x}^k) \rVert^2 \\
    &- \frac{\eta\alpha}{4} \lVert \overline{\nabla \f} (\x^k) \rVert^2 + \frac{\eta\alpha L^2}{2n} \lVert \bPhi^k \rVert^2.
\end{align}
Reorganize and lower bound the left-hand side to find
\begin{align*}
    \frac{\eta\alpha}{4} \hspace{-0.5mm} \big(  \hspace{-0.5mm} \lVert \nabla f (\bar{x}^k) \rVert^2  \hspace{-0.5mm} {+} \lVert \overline{\nabla \f} (\x^k) \rVert^2 \hspace{-0.3mm}  \big) & {\leq} f(\bar{x}^k) {-} f(\bar{x}^{k+1}){+}\frac{\eta\alpha L^2\lVert \bPhi^k \rVert^2}{2n} .
\end{align*}
Next, subtract and add $f^*$ and set $\tilde{f}(\bar{x}^k) = f(\bar{x}^k) - f^*$, then
\begin{align*}
    \lVert \nabla f (\bar{x}^k) \rVert^2 {+} \lVert \overline{\nabla \f} (\x^k) \rVert^2  & {\leq} \frac{4}{\eta\alpha} \big(f(\bar{x}^k) {-} f(\bar{x}^{k+1}) \big){+}\frac{2 L^2}{n} \lVert \bPhi^k \rVert^2.
\end{align*}
Sum both sides from $k = 0,...,K-1$ and divide by $K$ 
\begin{align*}
\frac{1}{K}\sum_{k=0}^{K-1} \biggl(\lVert &\nabla f (\bar{x}^k) \rVert^2 + \lVert \overline{\nabla \f} (\x^k) \rVert^2\biggl) \leq  \\
& \frac{4}{\eta\alpha K}\sum_{k=0}^{K-1} (\tilde{f}(\bar{x}^k) - \tilde{f}(\bar{x}^{k+1}) + \frac{2L^2}{nK} \sum_{k=0}^{K-1}\lVert \bPhi^k \rVert^2.
\end{align*}

Using Lemma~\ref{lem:consen_ineq}, we then have the following
\begin{align*}
&\frac{1}{K}\sum_{k=0}^{K-1} \left(\lVert \nabla f (\bar{x}^k) \rVert^2 + \lVert \overline{\nabla \f} (\x^k) \rVert^2\right) \leq \frac{4}{\eta\alpha K} \tilde{f}(\bar{x}^0) \\
&+\frac{2L^2(1-\bar{\lambda})(\sum_{k=0}^{K-1}\bar{\lambda}^{r_{k-1}+1})}{nK(1-\bar{\lambda}-e_1T_{\rm o})} \lVert \bPhi^0 \rVert^2 \\
&+\left( \frac{2L^2 e_2 T_{\rm o}}{nK(1 - \bar{\lambda} - e_1 T_{\rm o})} \right) \sum_{k=0}^{K-1} \left ( \lVert \overline{\nabla \mathbf{f}}(\mathbf{x}^k) \rVert^2 + \lVert \nabla f (\bar{x}^k) \rVert^2 \right).
\end{align*}
Therefore,
\begin{align*}
    &\left(1 - \frac{2L^2 e_2 T_{\rm o}}{n(1 - \bar{\lambda} - e_1 T_{\rm o})} \right)\frac{1}{K} \sum_{k=1}^{K-1} \left(\lVert \nabla f (\bar{x}^k) \rVert^2 + \lVert \overline{\nabla \f} (\x^k) \rVert^2\right)  \\
    &\leq \frac{4}{\eta\alpha K} \tilde{f}(\bar{x}^0) + \frac{2L^2(1-\bar{\lambda})\sum_{k=0}^{K-1}\bar{\lambda}^{r_{k-1}+1})}{nK(1-\bar{\lambda}-e_1T_{\rm o})} \lVert \bPhi^0 \rVert^2.
\end{align*}
Require 
\begin{equation*}
    \frac{1}{2}\leq \left(1 - \frac{2L^2 e_2 T_{\rm o}}{n(1 - \bar{\lambda} - e_1 T_{\rm o})} \right) \Rightarrow \alpha \leq \sqrt[4]{\frac{\lambda (1-\bar{\lambda})^2 (1-\theta)}{32L^4\eta^2T^2_{\rm o}}}.
\end{equation*}
Then, we have
\begin{align*}
    \frac{1}{K} \sum_{k=0}^{K-1} &\left(\lVert \nabla f (\bar{x}^k) \rVert^2 + \lVert \overline{\nabla \f} (\x^k) \rVert^2\right) \leq \frac{8}{\eta\alpha K} \tilde{f}(\bar{x}^0)  \\
    &+ \frac{4L^2(1-\bar{\lambda})(\sum_{k=0}^{K-1}\bar{\lambda}^{r_{k-1}+1})}{nK(1-\bar{\lambda}-e_1T_{\rm o})} \lVert \bPhi^0 \rVert^2.
\end{align*}
Assume that the initialization for $x_1, x_2, ..., x_n$ is identical. Then $\mathbf{x}^0 = \mathbf{1} \otimes x^0$ (for some $x^0 \in \mathbb{R}^d$). As a result, $\mathbf{x}^0 = \bar{\mathbf{x}}^0$ meaning $\lVert \hat{\mathbf{Q}}^T \mathbf{x}^0\rVert^2 = 0$. Then,
\begin{align*}
    \lVert \bPhi^0 \rVert^2 &=  \lVert \hat{\mathbf{Q}}^T \mathbf{y}^0 \rVert^2 = \left \lVert \alpha \hat{\mathbf{Q}}\tran \nabla \mathbf{f}(\mathbf{x}^0) \right \rVert^2 \\
    &= \alpha^2 \left \lVert \nabla \mathbf{f} (\bar{\mathbf{x}}^0)^T \hat{\mathbf{Q}} \hat{\mathbf{Q}}^T \hat{\mathbf{Q}} \hat{\mathbf{Q}}^T \nabla \mathbf{f}(\bar{\mathbf{x}}^0)\right \rVert^2 \\
    &= \alpha^2 \lVert \nabla \mathbf{f}(\bar{\mathbf{x}}^0) - \mathbf{1} \otimes \overline{\nabla \mathbf{f}}(\bar{\mathbf{x}}^0) \rVert^2.
\end{align*}
Define $\zeta_0 = \lVert \nabla \mathbf{f}(\bar{\mathbf{x}}^0) - \mathbf{1} \otimes \overline{\nabla \mathbf{f}}(\bar{\mathbf{x}}^0) \rVert^2$. We also upper bound $\sum_{k=0}^{K-1}\bar{\lambda}^{r_{k-1}+1}$ with $T_{\rm o}/(1-\bar{\lambda})$, a repeating geometric sequence, \cut{and the desired relation follows.}
\end{proof}

\section{Numerical Results} \label{sec:numerical}

\cut{We simulate the performance of Algorithm~\ref{alg:LU_GT} for the following  least squares problem with a non-convex regularization term~\cite{xin2021improved}:
\begin{align}\label{eq:simu_prob}
    \min_x \frac{1}{n} \sum_{i=1}^n \lVert A_i x - b_i \rVert^2 + \rho \sum_{j = 1}^d \frac{x(j)}{1+ x(j)},
\end{align}
where} $\{A_i, b_i\}$ is the local data held by agent $i$ and $x(j)$ is the \mbox{$j-th$} component of the parameter $x$.  In our particular simulation, $A_i \in \mathbb{R}^{p \times m}$ and $b_i \in \mathbb{R}^{p}$ where $p = 500, m = 20$. The values in $A_i$ are drawn from $\mathcal{N}(0, 1)$. A parameter vector $x^*_i \in \mathbb{R}^{20}$ is generated by $x^*_i + v_i$ where $x^*_i \sim \mathcal{N}(0, I_{m})$ and $v_i \sim \mathcal{N}(0, (0.1\times i)^2I_m)$ where $i$ is the agent index. Form $b_i = A_i x^i_0 + 0.1 \times z_i$ where $z_i \in \mathbb{R}^{500}$ is noise drawn from $\mathcal{N}(0, (0.1)^2)$. This is a heterogeneous case. We examine various topologies, including the 2D-MeshGrid, star, ring, and fully-connected graphs, each with 25 nodes. We also set $\rho = 0.01$.
\par

\begin{table}[h]
\caption{Manually optimized $\eta \alpha$ used for each graph and $T_{\rm o}$ combination.}
\centering
\resizebox{0.48\textwidth}{!}{
\begin{tabular}{l  l l l l l }
\toprule
                       & $T_{o} = 1$ & $T_{o} = 2$ &$T_{o} = 5$  & $T_{o} = 10$ & $T_{o} = 50$ \\ \toprule
\multicolumn{1}{l}{\textbf{Complete}} & $2\times 10^{-3}$ & $2\times 10^{-3}$ & $2\times 10^{-3}$ & $2\times 10^{-3}$ & $2\times 10^{-3}$ \\ \midrule
\multicolumn{1}{l}{\textbf{2D-Grid}} & $1\times 10^{-4}$ & $.5\times 10^{-4}$ & $.2\times 10^{-4}$ & $.1\times 10^{-4}$ & $.02\times 10^{-4}$  \\ \midrule
\multicolumn{1}{l}{\textbf{Ring}} & $2\times 10^{-5}$ & $1\times 10^{-5}$ & $0.4\times 10^{-5}$ & $.2\times 10^{-5}$ & $0.04\times 10^{-5}$ \\ \midrule
\multicolumn{1}{l}{\textbf{Star}} & $.4\times 10^{-4}$ & $.2\times 10^{-4}$ & $.08\times 10^{-4}$ & $.04\times 10^{-4}$ & $.008\times 10^{-4}$ \\ \bottomrule 
\end{tabular}
}
\label{table:step_size}
\end{table}

Table~\ref{table:step_size} lists the manually optimized $\eta \alpha$ for each graph and~$T_{\rm o}$ combination. Our simulation results in Figure~\ref{fig:LUGT} reveal that for fully-connected graphs LU-GT reduces communication costs. For sparse networks, the hyperparameter tuning of $\eta \alpha$ matches the suggested inversely proportional relation with $T_{\rm o}$ predicted by the theory. In this scenario, communication costs are equivalent to no local updates, matching the analysis.
\begin{figure}
\begin{tabular}{cc}
  \includegraphics[width=0.5\columnwidth]{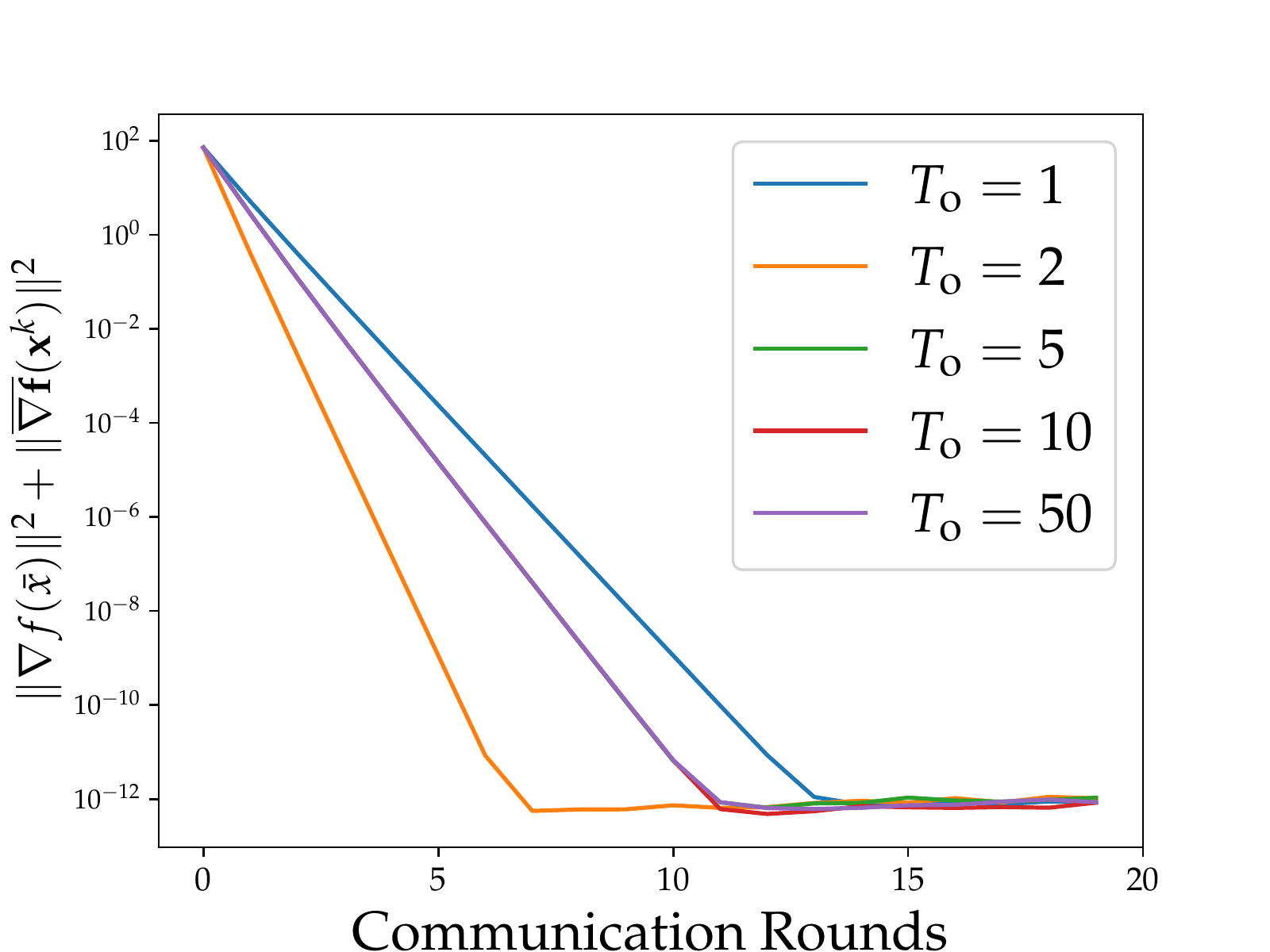} &   \includegraphics[width=0.5\columnwidth]{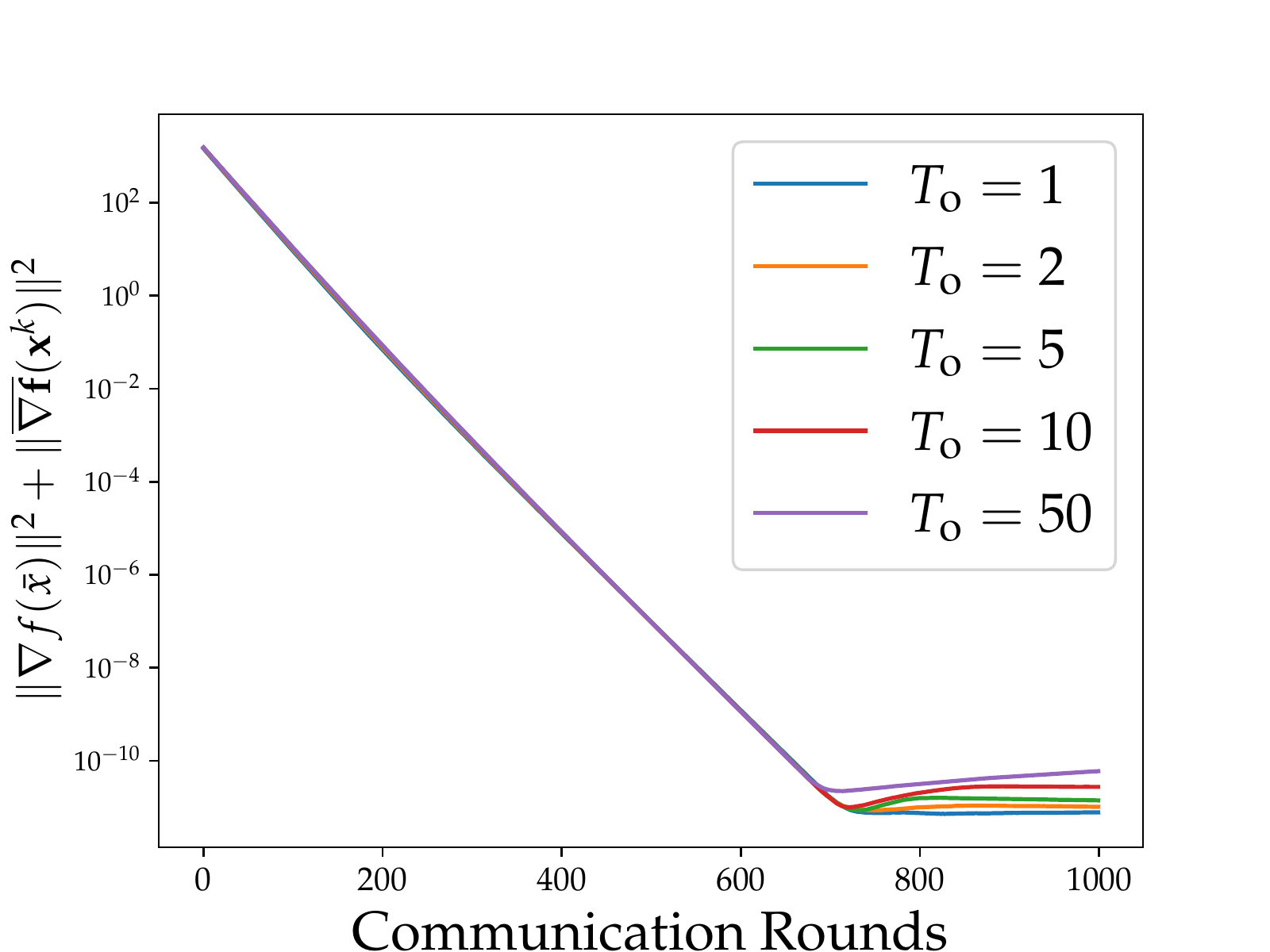} \\
(a) Fully-Connected Graph & (b) 2D-MeshGrid \\[6pt]
 \includegraphics[width=0.5\columnwidth]{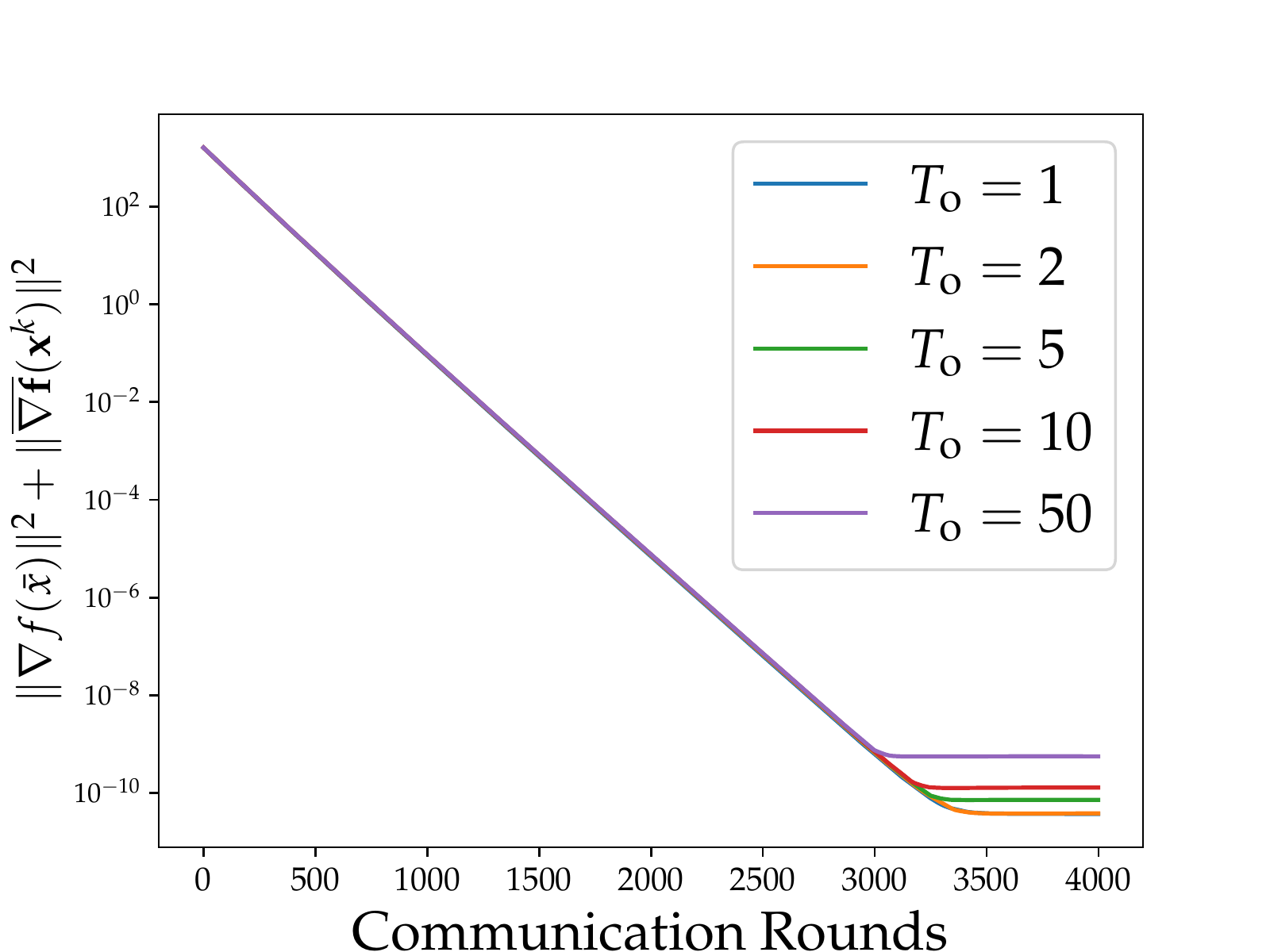} &   \includegraphics[width=0.5\columnwidth]{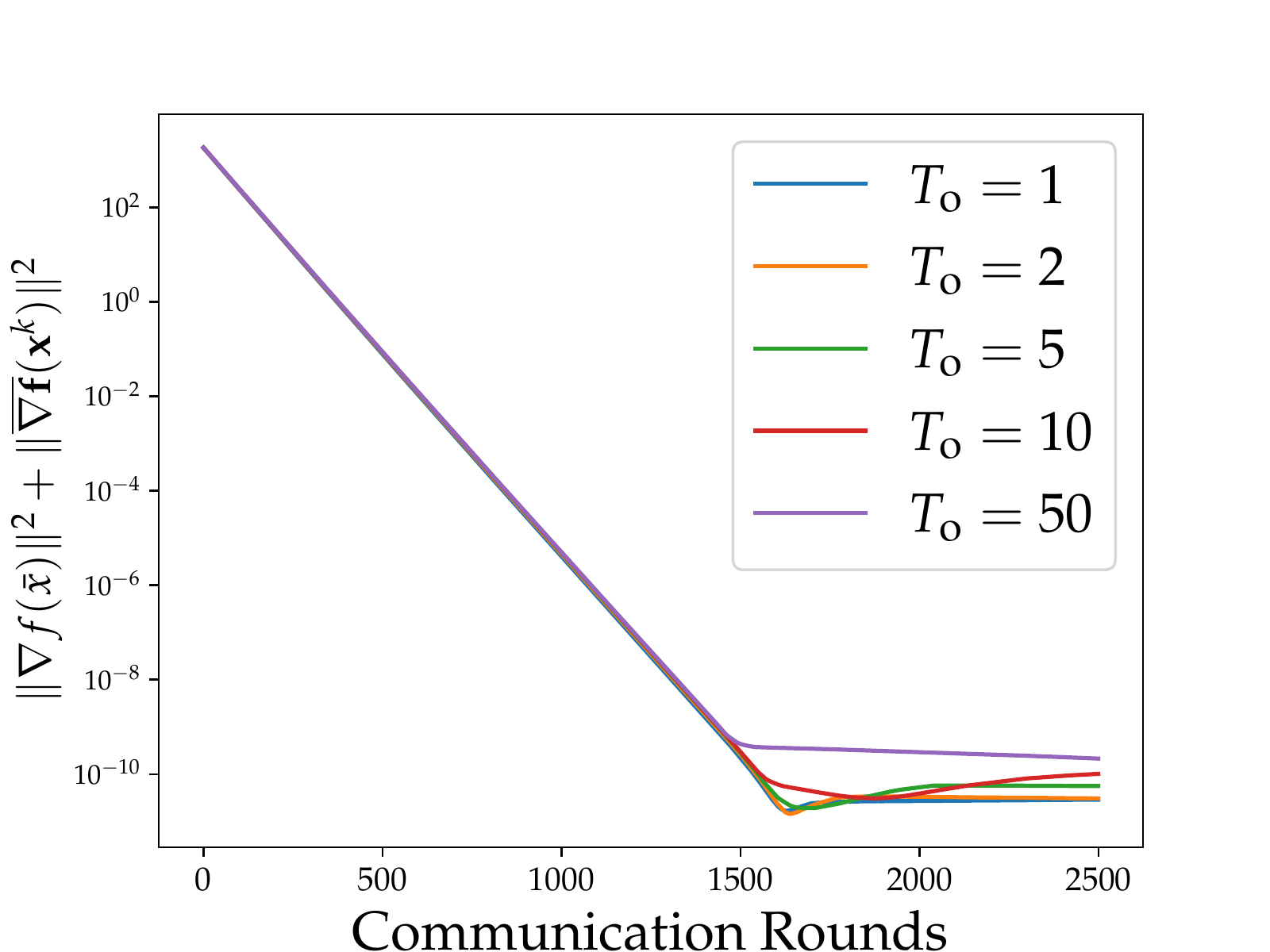} \\
(c) Ring & (d) Star \\[6pt]
\end{tabular}
\caption{Performance of LU-GT to solve~\eqref{eq:simu_prob} with varying $T_{\rm o }$, $\alpha \eta$, and topologies.}
\label{fig:LUGT}
\end{figure}

\section{Conclusions} \label{sec:concl}

We propose the algorithm LU-GT that incorporates local recursions into Gradient Tracking.  Our analysis shows that LU-GT matches the same communication complexity as the Federated Learning setting but allows arbitrary network topologies. In addition, regardless of the number of local recursions, LU-GT incurs no additional bias term in the rate. For well-connected graphs, communication complexity is reduced. Further refinement of the analysis is necessary to quantify the precise effect of local recursions on Gradient Tracking. It is still unclear under what regimes local updates reduce the communication cost and what the upper bound is on these local updates. Numerical analysis suggests that local updates might not benefit sparsely connected networks. Such explicit relations between network topologies and local updates are left for future work.












\bibliographystyle{IEEEtran}
\bibliography{References/IEEEabrv, References/myref}

\end{document}